\documentclass[pre, amsmath, amssymb, amsthm]{revtex4-2}

\usepackage{graphicx}
\usepackage{dcolumn}
\usepackage{bm}

\usepackage[utf8]{inputenc}
\usepackage[T1]{fontenc}
\usepackage{mathptmx}
\usepackage{etoolbox}

\makeatletter
\def\@email#1#2{%
 \endgroup
 \patchcmd{\titleblock@produce}
  {\frontmatter@RRAPformat}
  {\frontmatter@RRAPformat{\produce@RRAP{*#1\href{mailto:#2}{#2}}}\frontmatter@RRAPformat}
  {}{}
}%
\makeatother

\DeclareMathOperator{\Exp}{Exp} 
\DeclareMathOperator{\LN}{LN} 
\DeclareMathOperator{\Gam}{Gamma} 
\DeclareMathOperator{\GIG}{GIG} 
\DeclareMathOperator{\InvGamma}{IGamma} 
\DeclareMathOperator{\IG}{IG} 

\usepackage{amsthm, mathtools, amssymb}
\theoremstyle{plain}
\newtheorem{theorem}{Theorem}[section]
\newtheorem{lemma}{Lemma}[section]
\newtheorem{proposition}{Proposition}[section]
\newtheorem{corollary}{Corollary}[section]

\theoremstyle{definition}

\newtheorem{example}{Example}[section]

\theoremstyle{remark}
\newtheorem{remark}{Remark}[section]

\begin{document}


\title[Generalization of the double Pareto distribution]{Geometric Brownian motion with random observation time as generalization of the double Pareto distribution}
\author{K. Yamamoto}
 \email{yamamot@cs.u-ryukyu.ac.jp.}
 \affiliation{Faculty of Science, University of the Ryukyus, Nishihara, Okinawa 903-0213, Japan}
\author{T. Bando}%
\affiliation{ 
Production Engineering Department, SCM Division, Measurement Business Group, ANRITSU CORPORATION, Atsugi, Kanagawa 243-8555, Japan
}%

\author{H. Yanagawa}
\affiliation{%
Production Engineering Department, SCM Division, Measurement Business Group, ANRITSU CORPORATION, Atsugi, Kanagawa 243-8555, Japan
}%

\author{Y. Yamazaki}
\affiliation{%
School of Advanced Science and Engineering, Waseda University, Shinjuku, Tokyo 169-8555, Japan
}%

\date{}

\begin{abstract}
We study the probability distribution of the value of geometric Brownian motion at the stochastic observation time.
It is known that the exponentially distributed observation time yields the distribution called the double Pareto distribution, and this study aims to generalize this distribution.
First, we provide a calculation formula for the moment of the observed value of geometric Brownian motion using the moment-generating function of the observation time distribution.
Next, the probability density of the observed value of geometric Brownian motion is exactly derived under the observation time following the generalized inverse Gaussian distribution.
This result includes cases where the observation time follows the gamma, inverse gamma, and inverse Gaussian distributions, and can be regarded as a generalization of the double Pareto distribution.
\end{abstract}

\maketitle

\section{Introduction}
Stochastic differential equations have been extensively employed to describe and understand fluctuating or random phenomena, such as the Langevin equation in nonequilibrium statistical physics~\cite{Zwanzig}, the Schramm--Loewner evolution for critical phenomena~\cite{Katori}, and the Black--Scholes model in finance~\cite{Paul}.

The geometric Brownian motion~\cite{Oksendal} is a well-known stochastic process and is given by the stochastic differential equation
\begin{equation}
dS_t = \mu S_tdt+\sigma S_tdB_t,
\label{eq:GBM}
\end{equation}
where $B_t$ denotes Brownian motion and $\mu$ and $\sigma$ denote constants called drift and volatility in finance, respectively.
Stock prices in the Black--Scholes model are typically expressed by geometric Brownian motion~\cite{Paul}, and predictions of prices based on geometric Brownian motion have been attempted~\cite{Reddy, Abidin}.
In addition, the properties of geometric Brownian motion have been intensively investigated~\cite{Heston, Dufresne, Byczkowski, Biagini, Appleby}.

The solution of Eq.~\eqref{eq:GBM} under It\^o's interpretation is expressed as follows~\cite{Oksendal}:
\[
S_t=S_0\exp\left(\left(\mu-\frac{\sigma^2}{2}\right)t+\sigma B_t\right).
\]
For simplicity, we assume that $S_0$ is a positive constant and introduce $\tilde\mu\coloneqq \mu-\sigma^2/2$.
This solution implies that $S_t$ (at a given time $t$) follows the lognormal distribution $\LN(\ln S_0+\tilde\mu t, \sigma^2 t)$.
Here, a random variable $X$ is said to follow the lognormal distribution $\LN(\mu, \sigma^2)$ when the logarithm $\ln X$ is normally distributed with mean $\mu$ and variance $\sigma^2$, and the probability density function (PDF) of $X\sim\LN(\mu, \sigma^2)$ is given by
\[
f_{\LN}(x; \mu, \sigma^2)=\frac{1}{\sqrt{2\pi}\sigma x}\exp\left(-\frac{(\ln x-\mu)^2}{2\sigma^2}\right).
\]
See Crow and Shimizu~\cite{Crow} for the properties of the lognormal distribution.
The lognormal distribution has been observed in various natural and social phenomena~\cite{Limpert, Kobayashi, Yamamoto2016, Yamamoto2019}.

Suppose that the observation time is not a fixed constant $t$ but a random variable $T$, that is, the observation time varies from trial to trial, and we derive the distribution of $S_T$.
Let $G$ denote the cumulative distribution function of $T$.
The PDF of $S_T$ is expressed as
\begin{equation}
f_G(x)=\int_0^1 f_{\LN}(x; \ln S_0+\tilde\mu t, \sigma^2 t)dG(t).
\label{eq:general_f}
\end{equation}
In particular, when $T$ is exponentially distributed with the rate parameter $\lambda$, symbolically $T\sim\Exp(\lambda)$, the PDF of $S_T$ becomes
\[
f_{\Exp}(x)=\int_0^\infty \lambda e^{-\lambda t} f_\text{LN}(x;\ln S_0+\tilde\mu t, \sigma^2 t) dt
=\int_0^\infty \frac{\lambda e^{-\lambda t}}{\sqrt{2\pi t}\sigma x}\exp\left(-\frac{(\ln x-\ln S_0-\tilde\mu t)^2}{2\sigma^2t}\right)dt.
\]
This integral can be calculated using the following formula~\cite{Bateman}:
\begin{equation}
\int_0^\infty\frac{1}{\sqrt{t}}\exp\left(-pt-\frac{a}{4t}\right)dt=\sqrt{\frac{\pi}{p}}e^{-\sqrt{ap}}
\label{eq:DP_formula}
\end{equation}
for $a,p>0$.
This is the Laplace transform of the function $\exp(-a/(4t))/\sqrt{t}$ (from the function of $t$ to function of $p$), and this is a special case of Lemma~\ref{lemma1} stated in Section~\ref{sec3}.
Finally,
\begin{equation}
f_{\Exp}(x)=
\begin{dcases}
\frac{\lambda}{S_0\sqrt{\tilde{\mu}^2+2\sigma^2\lambda}}\left(\frac{x}{S_0}\right)^{-\alpha-1} & x\ge S_0,\\
\frac{\lambda}{S_0\sqrt{\tilde{\mu}^2+2\sigma^2\lambda}}\left(\frac{x}{S_0}\right)^{\beta-1} & x<S_0,
\end{dcases}
\label{eq:DP_f}
\end{equation}
is attained, where
\begin{equation}
\alpha=\frac{-\tilde\mu+\sqrt{\tilde\mu^2+2\sigma^2\lambda}}{\sigma^2}
\label{eq:DP_alpha}
\end{equation}
and
\begin{equation}
\beta=\frac{\tilde\mu+\sqrt{\tilde\mu^2+2\sigma^2\lambda}}{\sigma^2}.
\label{eq:DP_beta}
\end{equation}
The PDF $f_{\Exp}(x)$ is continuous over $x>0$ but not differentiable only at $x=S_0$.
The distribution of $S_T$ under $T\sim\Exp(\lambda)$ is referred to as the double Pareto distribution~\cite{Reed2004}; ``double Pareto'' means two power laws with exponents $-\alpha-1$ and $\beta-1$.
The double Pareto distribution is typically found in income~\cite{Reed2001} and microblog posting~\cite{Wang}, for example.
Other connections between the lognormal and power-law distributions can be found in a number of articles~\cite{Levy, Gabaix, Yamamoto2012, Yamamoto2014}.

The connection between geometric Brownian motion $S_t$ and the double Pareto distribution can also be formulated in the context of subordinated stochastic processes~\cite{}.
By introducing a stochastic clock $T(t)$, which is formally an increasing L\'evy process called the subordinator~\cite{Sato}, the subordinated process of $S_t$ is defined as the time-changed process $S_{T(t)}$.
An appropriate choice of $T(t)$ naturally yields anomalous diffusion~\cite{Metzler}, and subordinated processes have been applied to diverse systems such as nonequilibrium statistical physics~\cite{Barkai} and turbulence~\cite{Beck}.
When the subordinator $T(t)$ is what is called a gamma subordinator~\cite{Sato}, the PDF of $T(t)$ is
\[
g_{T(t)}(s)=\frac{\lambda^{kt}}{\Gamma(kt)}s^{kt-1}e^{-\lambda s},
\]
where $\Gamma$ is the gamma function.
That is, $T(t)$ follows the gamma distribution with the shape parameter $kt$ (time-dependent) and the rate parameter $\lambda$.
At time $t=k^{-1}$, i.e., $kt=1$, $T(k^{-1})$ follows the exponential distribution $\Exp(\lambda)$.
Thus, the double Pareto distribution can be regarded as the distribution of $S_{T(k^{-1})}$, where $T(t)$ is the gamma subordinator.

A possible extension of the double Pareto distribution is to consider $S_0$ a random variable.
In particular, when $S_0$ is lognormally distributed, the distribution of $S_T$ under $T\sim\Exp(\lambda)$ is called the double Pareto-lognormal distribution~\cite{Reed2004}.
The PDF of this distribution can be exactly calculated, having a complicated form compared with the double Pareto distribution~\eqref{eq:DP_f}.
The double Pareto-lognormal distribution has been used for human consumption~\cite{Toda} and telephone call~\cite{Seshadri}.
However, we do not treat random $S_0$ in this study and consider only the positive constant $S_0$.

In this study, we replace the exponential distribution of the observation time $T$ in the double Pareto distribution by the generalized inverse Gaussian (GIG) distribution.
We demonstrate that the PDF $f_{\GIG}(x)$ of $S_T$ can be expressed in a closed form by calculating the integral in Eq.~\eqref{eq:general_f} (see Theorem~\ref{thm2} in Section~\ref{sec3}).
The GIG distribution includes the exponential distribution as a special case.
Therefore, this result constitutes a generalization of the double Pareto distribution.
However, due to the complicated form of $f_{\GIG}(x)$, its cumulative distribution function $\int_0^x f_{\GIG}(y)dy$ does not appear to be expressed using existing special functions.
In this sense, we believe that the GIG distribution provides the maximal generalization of the double Pareto distribution that allows for the exact calculation of the PDF.
As special cases of the GIG distribution, the PDFs of $S_T$ are derived when $T$ follows the gamma, inverse gamma, and inverse Gaussian distributions, and their properties are investigated.
Prior to Theorem~\ref{thm2}, we provide in Theorem~\ref{thm1} the formula for the moment of $S_T$.
Using this formula, the $m$th moment $E[S_T^m]$ for $m=0,1,2,\ldots$ can be obtained even if the integral in Eq.~\eqref{eq:general_f} cannot be calculated and the distribution of $S_T$ is unknown.
For convenience, the list of notation used in this paper is summarized in Table~\ref{tbl1}.

As a related study, the authors recently derived the PDF and cumulative distribution function of $S_T$ when $T$ follows a continuous uniform distribution~\cite{Yamamoto2024}.
This previous study shares the question of what distribution of $T$ except for the exponential distribution yields an exact distribution of $S_T$.
However, this result is not a direct generalization of the double Pareto distribution, because the family of continuous uniform distributions does not contain exponential distributions.
Additionally, in relation to subordinated processes, $S_T$ investigated in this study corresponds to the subordinated process $S_{T(t)}$ at $t=1$ in which $T(t)$ is the GIG subordinator~\cite{Nielsen}.

\begin{table}[t]\centering
\caption{
List of notation used in this paper.
}
\begin{tabular}{ll}
\toprule
Symbol & Description \\
\colrule
$S_t$ & Geometric Brownian motion \\
$\mu$, $\sigma$ & Parameters (drift and volatility) of geometric Brownian motion\\
$\tilde\mu$ & $\tilde\mu\coloneqq\mu-\sigma^2/2$ \\
$T$ & Random observation time \\
$\lambda$ & Rate parameter of exponential and gamma distributions \\
$M_T(s)$ & Moment-generating function of $T$ \\
$\nu$, $\psi$, $\chi$ & Parameters of GIG distribution \\
$k$ & Shape parameter of gamma and inverse gamma distributions \\
$\theta$ & Scale parameter of inverse gamma distribution \\
$\tau$, $\omega$ & Parameters (mean and shape parameter) of inverse Gaussian distribution\\
\botrule
\end{tabular}
\label{tbl1}
\end{table}

\section{Calculation of moments for general observation time distribution}
In the calculation of the moment of $S_T$, i.e., $E[S_T^m]$ for $m=0,1,2,\ldots$, the moment of a lognormal variable is required.
\begin{proposition}[Moment of lognormal~\cite{Crow}]
The $m$th moment of the lognormal random variable $X\sim\LN(\mu, \sigma^2)$ is given by
\begin{equation}
E[X^m]=\exp\left(m\mu+\frac{m^2}{2}\sigma^2\right).
\label{eq:lognormal_moment}
\end{equation}
\label{prop1}
\end{proposition}

\begin{proof}
Proved by directly calculating the integral
\[
E[X^m]=\int_0^\infty x^m f_{\LN}(x; \mu, \sigma^2) dx=\int_0^\infty \frac{x^m}{\sqrt{2\pi}\sigma x}\exp\left(-\frac{(\ln x-\mu)^2}{2\sigma^2}\right)dx.
\]
\end{proof}

The mean and variance of $X\sim\LN(\mu, \sigma^2)$ are immediately calculated using Proposition~\ref{prop1}, as
\[
E[X] = \exp\left(\mu+\frac{\sigma^2}{2}\right),\quad
V[X] = e^{2\mu+\sigma^2}(e^{\sigma^2}-1),
\]
respectively.

\begin{theorem}[Moment formula for $S_T$]
The $m$th moment of $S_T$ becomes
\begin{equation}
E[S_T^m] = S_0^m M_T\left(m\tilde\mu+\frac{m^2}{2}\sigma^2\right)\quad (m=0,1,2,\ldots),
\label{eq:thm1}
\end{equation}
where $M_T(s)\coloneqq E[\exp(sT)]$ is the moment-generating function of the observation time $T$.
The moment $E[S_T^m]$ diverges if $s=m\tilde\mu+m^2\sigma^2/2$ lies outside the domain of $M_T(s)$.
\label{thm1}
\end{theorem}

\begin{proof}
We let $G$ denote the cumulative distribution function of $T$.
As the PDF of $S_T$ is given by $f_G(x)$ in Eq.~\eqref{eq:general_f}, the $m$th moment of $S_T$ is calculated to be
\begin{align*}
E[S_T^m]&=\int_0^\infty x^m f_G(x) dx\\
&=\int_0^1 \left(\int_0^\infty x^m f_\text{LN}(x; \ln S_0+\tilde\mu t,\sigma^2 t)dx\right) dG(t)\\
&=S_0^m\int_0^1 \exp\left(m\tilde\mu t+\frac{m^2}{2}\sigma^2 t\right)dG(t)\\
&=S_0^m M_T\left(m\tilde\mu+\frac{m^2}{2}\sigma^2\right).
\end{align*}
We note that Eq.~\eqref{eq:lognormal_moment} for the $m$th moment of the lognormal distribution is applied to calculate the integral of $x$.
\end{proof}

By substituting $m=1$ and $2$ into Eq.~\eqref{eq:thm1}, the mean and variance of $S_T$ are immediately obtained:
\begin{align*}
E[S_T]&=S_0 M_T\left(\tilde\mu+\frac{\sigma^2}{2}\right)=S_0 M_T(\mu),\\
V[S_T]&=S_0^2[M_T(2\tilde\mu+2\sigma^2)-M_T(\mu)^2]=S_0^2[M_T(2\mu+\sigma^2)-M_T(\mu)^2].
\end{align*}

The important point of Theorem~\ref{thm1} is that even if the distribution of $S_T$ is unknown, the moment $E[S_T^k]$ can be calculated from the moment-generating function of $T$ (not of $S_T$).

We provide examples of the moment calculation.
\begin{example}
When $T\sim\Exp(\lambda)$, $S_T$ follows the double Pareto distribution given by Eq.~\eqref{eq:DP_f}.
The moment-generating function of $T$ is $M_T(s)=\lambda/(\lambda-s)$, defined within $s<\lambda$.
From Theorem~\ref{thm1}, the $m$th moment $E[S_T^m]$ is finite when $m\tilde\mu+m^2\sigma^2/2<\lambda$.
By solving this inequality with respect to $m$, we find that $E[S_T^m]$ becomes finite when the order $m$ satisfies
\[
m<\frac{-\tilde\mu+\sqrt{\tilde\mu^2+2\sigma^2\lambda}}{\sigma^2}=\alpha,
\]
where $\alpha$ is introduced in Eq.~\eqref{eq:DP_alpha}.
Using Eq.~\eqref{eq:thm1}, the $m$th moment of $S_T$ for this $m$ value becomes
\[
E[S_T^m]=\frac{\lambda S_0^m}{\lambda-m\tilde\mu-m^2\sigma^2/2}.
\]
When $\lambda>\mu$, the mean of $S_T$ becomes
\[
E[S_T]=\frac{\lambda S_0}{\lambda-\mu},
\]
and when $\lambda>2\mu+\sigma^2$, the variance of $S_T$ becomes
\[
V[S_T]=\frac{\lambda S_0^2}{\lambda-2\mu-\sigma^2}-\left(\frac{\lambda S_0}{\lambda-\mu}\right)^2.
\]
Although these results can be obtained by directly calculating integrals involving Eq.~\eqref{eq:DP_f}, Theorem~\ref{thm1} allows significantly easier computations.
\label{ex1}
\end{example}

\begin{example}
When $T$ takes a constant $t_0>0$ almost surely, the moment-generating function of $T$ becomes $M_T(s)=\exp(t_0 s)$.
Using Eq.~\eqref{eq:thm1},
\[
E[S_T^m]=S_0^m \exp\left(m\tilde\mu t_0+\frac{m^2}{2}\sigma^2 t_0\right)=\exp\left(m(\ln S_0+\tilde\mu t_0)+\frac{m^2}{2}\sigma^2 t_0\right),
\]
which is identical to the $m$th moment of the lognormal distribution $\LN(\ln S_0+\tilde\mu t_0, \sigma^2 t_0)$.
This result is consistent with the property that $S_t$ at $t=t_0$ follows $\LN(\ln S_0+\tilde\mu t_0, \sigma^2 t_0)$.
\label{ex2}
\end{example}

\begin{example}
When $T$ follows a continuous uniform distribution on the interval $[0,t_\text{max}]$, the moment-generating function is $M_T(s)=(e^{t_\text{max}s}-1)/(t_\text{max}s)$ defined for all $-\infty<s<\infty$.
From Theorem~\ref{thm1}, the moment of $S_T$ in this case becomes
\[
E[S_T^m]=S_0^m\frac{2}{2m\tilde\mu t_\text{max}+m^2\sigma^2 t_\text{max}}\left[\exp\left(m\tilde\mu t_\text{max}+\frac{m^2\sigma^2 t_\text{max}}{2}\right)-1\right].
\]
Unlike Examples~\ref{ex1} and \ref{ex2}, the moment of $S_T$ of each order $m$ is finite for any parameter values.
This result is the same as reported in our previous paper~\cite{Yamamoto2024}, but Theorem~\ref{thm1} simplifies the calculation.
\end{example}

\section{Size distribution when observation time follows the GIG distribution}\label{sec3}
The moment of $S_T$ can be calculated using Eq.~\eqref{eq:thm1} for general $T$.
However, the PDF of $S_T$ is difficult to obtain because we must calculate the integral in Eq.~\eqref{eq:general_f}.
In this section, we show that the GIG distribution is a family of probability distributions for $T$ where the PDF of $S_T$ can be exactly calculated.

The GIG distribution $\GIG(\nu, \psi,\chi)$ is a three-parameter distribution ($\psi>0$, $\chi>0$, and $-\infty<\nu<\infty$), whose PDF is given by
\begin{equation}
g_{\GIG}(t; \nu, \psi,\chi)=\frac{(\psi/\chi)^{\nu/2}}{2K_\nu(\sqrt{\psi\chi})}t^{\nu-1}\exp\left(-\frac{\psi t}{2}-\frac{\chi}{2t}\right),
\label{eq:GIG_PDF}
\end{equation}
where $K_\nu$ denotes the modified Bessel function of the second kind~\cite{Olver}.
According to an encyclopedia~\cite{ESS}, this distribution was originally introduced by Halphen to analyze a monthly water flow in hydroelectric stations.
For the properties of the GIG distribution, refer to Refs.~\cite{Koudou, Jorgensen, Paolella}.
As subsequently explained, the GIG distribution includes gamma, inverse gamma, and inverse Gaussian distributions.

The following integral is essential for calculating the PDF $f_{\GIG}(x)$ of $S_T$ when $T$ follows the GIG distribution.
\begin{lemma}
For $a,p>0$,
\[
\int_0^\infty t^{\nu-1}\exp\left(-pt-\frac{a}{4t}\right)dt=2\left(\frac{a}{4p}\right)^{\nu/2}K_\nu(\sqrt{ap}).
\]
\label{lemma1}
\end{lemma}

This integral is the Laplace transform of the function $t^{\nu-1}\exp(-a/(4t))$ (from a function of $t$ to a function of $p$) and can be regarded as a generalization of Eq.~\eqref{eq:DP_formula}.
A standard proof of Lemma~\ref{lemma1} involves a contour integral in the complex plane~\cite{Gray}.
Although $K_\nu(z)$ is defined for complex values of $z$ and $\nu$, we consider only positive $z$ and real $\nu$ in this study, for which $K_\nu(z)$ always takes real values.
Further properties and discussions for $K_\nu(z)$ are presented in Appendix~\ref{apdx}.

Using Lemma~\ref{lemma1}, we obtain the PDF of $S_T$ in which $T$ follows the GIG distribution, which is the second main result of this study.
\begin{theorem}
The PDF of $S_T$ with $T\sim\GIG(\nu,\psi,\chi)$ becomes
\begin{align*}
f_{\GIG}(x)&=\frac{(\psi/\chi)^{\nu/2}}{\sqrt{2\pi}\sigma K_\nu(\sqrt{\psi\chi})S_0}\left(\frac{x}{S_0}\right)^{\tilde\mu/\sigma^2-1}\left(\frac{(\ln(x/S_0))^2}{\sigma^2}+\chi\right)^{(2\nu-1)/4}\left(\frac{\tilde\mu^2}{\sigma^2}+\psi\right)^{-(2\nu-1)/4}\\
&\quad K_{\nu-1/2}\left(\sqrt{\left(\frac{(\ln(x/S_0))^2}{\sigma^2}+\chi\right)\left(\frac{\tilde\mu^2}{\sigma^2}+\psi\right)}\right).
\end{align*}
\label{thm2}
\end{theorem}

\begin{proof}
Substituting Eq.~\eqref{eq:GIG_PDF} into Eq.~\eqref{eq:general_f}, we obtain
\begin{align*}
f_{\GIG}(x)&=\int_0^\infty f_{\LN}(x; \ln S_0+\tilde\mu t, \sigma^2t) g_{\GIG}(t; \nu, \psi,\chi)dt\\
&=\frac{(\psi/\chi)^{\nu/2}}{2\sqrt{2\pi}\sigma K_\nu(\sqrt{\psi\chi})x}\left(\frac{x}{S_0}\right)^{\tilde\mu/\sigma^2}\int_0^\infty t^{\nu-3/2} \exp\left(-\frac{1}{2}\left(\frac{\tilde\mu^2}{\sigma^2}+\psi\right)t-\frac{1}{2t}\left(\frac{(\ln(x/S_0))^2}{\sigma^2}+\chi\right)\right)dt.
\end{align*}
The proof is completed by calculating the integral using Lemma~\ref{lemma1}.
\end{proof}

\begin{remark}
Due to the complicated form of $f_{\GIG}(x)$, its cumulative distribution function $\int_0^x f_{\GIG}(y)dy$ does not appear to be expressed using existing special functions, as far as the authors have attempted to calculate.
Hence, the authors think that the GIG distribution is a marginal class for the distribution of $T$ such that the PDF of $S_T$ can be exactly calculated.
\end{remark}

\subsection{Observation time with gamma distribution}
As a special case of the GIG distribution, we substitute $\nu=k$, $\psi=2\lambda$, and $\chi\to0+$ in $g_{\GIG}(t; \nu,\psi,\chi)$ in Eq.~\eqref{eq:GIG_PDF}.
Caution is required when taking the $\chi\to0+$ limit.
Using Lemma~\ref{lemma1} with $p=\psi$ and $a=\chi$, Eq.~\eqref{eq:GIG_PDF} can be expressed as follows:
\begin{equation}
g_{\GIG}(t;\nu,\psi,\chi)=2^{-\nu}t^{\nu-1}\exp\left(-\frac{\psi t}{2}-\frac{\chi}{2t}\right)
\left[\int_0^\infty u^{\nu-1}\exp\left(-\psi u-\frac{\chi}{4u}\right)du\right]^{-1}.
\label{eq:gamma_PDF}
\end{equation}
This expression allows us to take the $\chi\to0+$ limit directly and obtain
\[
g_{\GIG}(t; k, 2\lambda, 0)=\frac{\lambda^k}{\Gamma(k)}t^{k-1}e^{-\lambda t}.
\]
This is the PDF of the gamma distribution with the shape parameter $k$ and the rate parameter $\lambda$. (The scale parameter is $1/\lambda$.)
In this study, we use $\Gam(k, \lambda)$ for this gamma distribution.
The gamma distribution is reduced to an exponential distribution when $k=1$, i.e., $\GIG(1,2\lambda,0)=\Gam(1,\lambda)=\Exp(\lambda)$.
Another special case $\Gam(\nu/2,2)$ is known as the chi-squared distribution.
The gamma distribution appears in a great variety of phenomena~\cite{Sharma, Wright, Park, Yamamoto2024_1}.

As a corollary of Theorem~\ref{thm2}, the PDF of $S_T$ under $T\sim\Gam(k, \lambda)$ is obtained.
\begin{corollary}
The PDF of $S_T$ with $T\sim\Gam(k, \lambda)$ becomes
\begin{equation}
f_{\Gam}(x)=\frac{2\lambda^k}{\sqrt{2\pi}\Gamma(k)\sigma(\tilde\mu^2+2\sigma^2\lambda)^{(2k-1)/4} S_0}\left(\frac{x}{S_0}\right)^{\tilde\mu/\sigma^2-1}\left|\ln\frac{x}{S_0}\right|^{k-1/2}K_{k-1/2}\left(\frac{\sqrt{\tilde\mu^2+2\sigma^2\lambda}}{\sigma^2}\left|\ln\frac{x}{S_0}\right|\right).
\label{eq:gamma_f}
\end{equation}
\end{corollary}

\begin{proof}
The PDF $f_{\Gam}(x)$ is obtained by substituting $\nu=k$, $\psi=2\lambda$, and $\chi\to0+$ into Theorem~\ref{thm2}.
In taking the $\chi\to0+$ limit, we use Proposition~\ref{prop:K}(c) in Appendix~\ref{apdx} with $\nu=k>0$ as
\begin{equation}
\frac{\chi^{-\nu/2}}{K_\nu(\sqrt{\psi\chi})}\simeq\frac{2\chi^{-\nu/2}}{\Gamma(\nu)(\sqrt{\psi\chi}/2)^{-\nu}}=\frac{\psi^{\nu/2}}{2^{\nu-1}\Gamma(\nu)}.
\label{eq:gamma_proof}
\end{equation}
\end{proof}

By substituting $k=1$ and using Proposition~\ref{prop:K}(d) regarding $K_{1/2}(z)$, $f_{\Gam}(x)$ is reduced to the double Pareto distribution in Eq.~\eqref{eq:DP_f}.
Therefore, $f_{\Gam}(x)$ is a generalization of the double Pareto distribution.

\begin{remark}\label{remark:gamma}
Equation~\eqref{eq:gamma_f} can be expressed as
\begin{align*}
f_{\Gam}(x)&=\frac{2(\lambda/\sigma^2)^k}{\sqrt{2\pi}\Gamma(k)(2(\lambda/\sigma^2)+(\tilde\mu/\sigma^2)^2)^{(2k-1)/4}S_0}\left(\frac{x}{S_0}\right)^{\tilde\mu/\sigma^2-1}\left|\ln\frac{x}{S_0}\right|^{k-1/2}\\
&\quad K_{k-1/2}\left(\sqrt{2\frac{\lambda}{\sigma^2}+\left(\frac{\tilde\mu}{\sigma^2}\right)^2}\left|\ln\frac{x}{S_0}\right|\right).
\end{align*}
Although $f_{\Gam}(x)$ has five parameters $k$, $\lambda$, $\tilde\mu$, $\sigma$, and $S_0$, we can reduce them to four parameters $k$, $\lambda/\sigma^2$, $\tilde\mu/\sigma^2$, and $S_0$.
Moreover, $S_0$ can be eliminated by considering $x/S_0$ instead of $x$.
\end{remark}

The moment-generating function of $T\sim\Gam(k,\lambda)$ is
\[
M_T(s)=\left(\frac{\lambda}{\lambda-s}\right)^k
\]
for $s<\lambda$.
Owing to Theorem~\ref{thm1},
\[
E[S_T^m]=S_0^m\left(\frac{\lambda}{\lambda-m\tilde\mu-m^2\sigma^2/2}\right)^k,
\]
and this moment becomes finite if
\[
m\tilde\mu+\frac{m^2}{2}\sigma^2<\lambda.
\]
The shape parameter $k$ does not affect whether the moment is finite.
The mean and variance of $S_T$ become
\[
E[S_T]=S_0\left(\frac{\lambda}{\lambda-\mu}\right)^k,\quad
V[S_T]=S_0^2\left[\left(\frac{\lambda}{\lambda-2\mu-\sigma^2}\right)^k-\left(\frac{\lambda}{\lambda-\mu}\right)^{2k}\right].
\]

Using the asymptotic form of $K_\nu(z)$ presented in Proposition~\ref{prop:K}(b) and (c) in Appendix~\ref{apdx}, we obtain the following asymptotic forms of $f_{\Gam}(x)$.
\begin{proposition}[Asymptotic form of $f_{\Gam}(x)$]
\begin{enumerate}
\renewcommand{\labelenumi}{$\mathrm{(\alph{enumi})}$}
\item Asymptotic form in the limit $x\to\infty$:
\[
f_{\Gam}(x)\simeq\frac{\lambda^k}{\Gamma(k)(\tilde\mu^2+2\sigma^2\lambda)^{k/2}S_0}\left(\frac{x}{S_0}\right)^{-\alpha-1}\left(\ln\frac{x}{S_0}\right)^{k-1} \quad(x\to\infty),
\]
where $\alpha$ is the same exponent as that given in Eq.~\eqref{eq:DP_alpha}.
\item Asymptotic form in the limit $x\to0+$:
\[
f_{\Gam}(x)\simeq\frac{\lambda^k}{\Gamma(k)(\tilde\mu^2+2\sigma^2\lambda)^{k/2}S_0}\left(\frac{x}{S_0}\right)^{\beta-1}\left|\ln\frac{x}{S_0}\right|^{k-1} \quad(x\to0+),
\]
where the exponent $\beta$ is given by Eq.~\eqref{eq:DP_beta}.
\item Asymptotic form in the limit $x\to S_0$:
\[
f_{\Gam}(x)\simeq
\begin{dcases}
\frac{2^{k-1}\lambda^k\sigma^{2(k-1)}\Gamma(k-1/2)}{\sqrt{\pi}\Gamma(k)(\tilde\mu^2+2\sigma^2\lambda)^{k-1/2}S_0} & k>\frac{1}{2},\\
\frac{2^{-k}\lambda^k\Gamma(1/2-k)}{\sqrt{\pi}\Gamma(k)\sigma^{2k}S_0}\left|\ln\frac{x}{S_0}\right|^{2k-1} & k<\frac{1}{2},\\
-\frac{\sqrt{2\lambda}}{\pi\sigma S_0}\ln\left|\ln\frac{x}{S_0}\right| & k=\frac{1}{2}
\end{dcases}
\quad(x\to S_0).
\]
Therefore, $f_{\Gam}(x)$ is finite at $x=S_0$ if $k>1/2$, and $\lim_{x\to S_0}f_{\Gam}(x)$ diverges if $k\le1/2$, irrespective of the other parameters $\tilde\mu$, $\sigma$, $\lambda$, and $S_0$.
\end{enumerate}
\label{prop:gamma}
\end{proposition}

\begin{proof}
Both (a) and (b) can be obtained by using Proposition~\ref{prop:K}(b) in Appendix~\ref{apdx} because $x\to\infty$ and $x\to0+$ correspond to $|\ln(x/S_0)|\to\infty$.
In the computation, we use
\[
\exp\left(-\frac{\sqrt{\tilde\mu^2+2\sigma^2\lambda}}{\sigma^2}\left|\ln\frac{x}{S_0}\right|\right)
=\begin{dcases}
\left(\frac{x}{S_0}\right)^{-\sqrt{\tilde\mu^2+2\sigma^2\lambda}/\sigma^2} & x>S_0,\\
\left(\frac{x}{S_0}\right)^{\sqrt{\tilde\mu^2+2\sigma^2\lambda}/\sigma^2} & x<S_0.
\end{dcases}
\]
We apply the $x>S_0$ case to the $x\to\infty$ limit and the $x<S_0$ case to the $x\to0+$ limits, and this difference yields the different exponents $\alpha$ and $\beta$ in (a) and (b).

(c) can be obtained from Proposition~\ref{prop:K}(c).
However, we cannot apply it directly when $k<1/2$ ($k-1/2<0$) and need to use Proposition~\ref{prop:K}(a) as $K_{k-1/2}=K_{-k+1/2}$ in this case.
\end{proof}

From Proposition~\ref{prop:gamma}(b), $\lim_{x\to0+}f_{\Gam}(x)=0$ for $\beta>1$ and $\lim_{x\to0+}f_{\Gam}(x)=\infty$ for $\beta<1$.
The inequality $\beta\gtrless 1$ is equivalent to
\begin{equation}
\frac{\tilde\mu}{\sigma^2}+\frac{\lambda}{\sigma^2}\gtrless \frac{1}{2}.
\label{eq:gamma_inequality}
\end{equation}
When $\beta=1$, $\lim_{x\to0+}f_{\Gam}(x)=0$ for $k>1$, $\lim_{x\to0+}f_{\Gam}(x)=\infty$ for $k<1$, and
\[
f_{\Gam}(0)=\frac{\lambda}{\sqrt{\tilde\mu^2+2\sigma^2\lambda}S_0}
\]
for $k=1$.

The differentiability of $f_{\Gam}(x)$ at $x=S_0$ is unclear immediately, in that $|\ln(x/S_0)|$ is not differentiable at $x=S_0$ and $K_{k-1/2}$ in Eq.~\eqref{eq:gamma_f} diverges at $x=S_0$ (the argument of $K_{k-1/2}$ becomes $0$).
We use symbols $f_{\Gam}'(S_0^{+})$ and $f_{\Gam}'(S_0^{-})$ for the right and left derivatives at $x=S_0$, respectively.
We do not consider the differentiability for the $k\le1/2$ case because $f_{\Gam}(x)$ diverges as $x\to S_0$ [see Proposition~\ref{prop:gamma}(c)].
\begin{proposition}[Differentiability of $f_{\Gam}(x)$ at $x=S_0$]
The differentiability of $f_{\Gam}(x)$ at $x=S_0$ varies depending on $k$ as follows:
\begin{enumerate}
\renewcommand{\labelenumi}{$\mathrm{(\roman{enumi})}$}
\item For $1/2<k<1$, both $f_{\Gam}'(S_0^{+})$ and $f_{\Gam}'(S_0^{-})$ diverge to infinity; $f_{\Gam}'(S_0^{\pm})=\mp\infty$.
\item For $k=1$, $f_{\Gam}'(S_0^{\pm})$ is finite but $f_{\Gam}'(S_0^{+})\ne f_{\Gam}'(S_0^{-})$.
\item For $k>1$, $f_{\Gam}'(S_0^{+})=f_{\Gam}'(S_0^{-})$; that is, $f_{\Gam}(x)$ is differentiable at $x=S_0$.
\end{enumerate}
\label{prop:gamma_diff}
\end{proposition}

\begin{proof}
Depending on $x\gtrless S_0$, $|\ln(x/S_0)|=\pm\ln(x/S_0)$.
In this proof, wherever a plus--minus  or minus--plus sign occurs, the upper and lower signs refer to $x>S_0$ and $x<S_0$, respectively.
The derivative of $f_{\Gam}(x)$ at $x\ne S_0$ becomes
\begin{align*}
f_{\Gam}'(x)&=\frac{2\lambda^k}{\sqrt{2\pi}\Gamma(k)\sigma(\tilde\mu^2+2\sigma^2\lambda)^{(2k-1)/4}}
\frac{1}{x}\left(\frac{x}{S_0}\right)^{\tilde\mu/\sigma^2-1}\left(\pm\ln\frac{x}{S_0}\right)^{k-1/2}\\
&\qquad\left[\left(\frac{\tilde\mu}{\sigma^2}-1\right)K_{k-1/2}\left(\pm\frac{\sqrt{\tilde\mu^2+2\sigma^2\lambda}}{\sigma^2}\ln\frac{x}{S_0}\right)
\mp\frac{\sqrt{\tilde\mu^2+2\sigma^2\lambda}}{\sigma^2}K_{k-3/2}\left(\pm\frac{\sqrt{\tilde\mu^2+2\sigma^2\lambda}}{\sigma^2}\ln\frac{x}{S_0}\right)
\right].
\end{align*}
By using Proposition~\ref{prop:K}(e) in Appendix~\ref{apdx}, the derivative $K_{k-1/2}'$ can be expressed using $K_{k-1/2}$ and $K_{k-3/2}$.
In the $x\to S_0\pm$ limit, the argument of Bessel functions approaches $0+$.
To apply Proposition~\ref{prop:K}(c) to calculate $f_\text{Gamma}'(S_0^{\pm})$, it should be noted that the asymptotic form of $K_{k-3/2}$ changes depending on the magnitude of $k$ and $3/2$.

When $k>3/2$, the cases $x>S_0$ and $x<S_0$ have the same form:
\begin{align*}
f_{\Gam}'(x)&\simeq\frac{2\lambda^k}{\sqrt{2\pi}\Gamma(k)\sigma(\tilde\mu^2+2\sigma^2\lambda)^{(2k-1)/4}}
\frac{1}{x}\left(\frac{x}{S_0}\right)^{\tilde\mu/\sigma^2-1}\\
&\qquad\left[\frac{1}{2}\left(\frac{\tilde\mu}{\sigma^2}-1\right)\Gamma\left(k-\frac{1}{2}\right)\left(\frac{\sqrt{\tilde\mu^2+2\sigma^2\lambda}}{2\sigma^2}\right)^{-k+1/2}-\Gamma\left(k-\frac{3}{2}\right)\left(\frac{\sqrt{\tilde\mu^2+2\sigma^2\lambda}}{2\sigma^2}\right)^{-k+5/2}\ln\frac{x}{S_0}\right]\\
& \quad(x\to S_0\pm).
\end{align*}
Hence, $f_{\Gam}(x)$ is differentiable at $x=S_0$ and the derivative becomes
\[
f_{\Gam}'(S_0)=\frac{2^{k-1}\Gamma(k-1/2)\sigma^{2k-2}\lambda^k}{\sqrt{\pi}S_0\Gamma(k)(\tilde\mu^2+2\sigma^2\lambda)^{k-1/2}}\left(\frac{\tilde\mu}{\sigma^2}-1\right).
\]

When $k=3/2$,
\begin{align*}
f_{\Gam}'(x)&\simeq\frac{2\lambda^{3/2}}{\sqrt{2\pi}\Gamma(3/2)\sigma(\tilde\mu^2+2\sigma^2\lambda)^{1/2}}
\frac{1}{x}\left(\frac{x}{S_0}\right)^{\tilde\mu/\sigma^2-1}\\
&\qquad\left[\left(\frac{\tilde\mu}{\sigma^2}-1\right)\frac{\sigma^2}{\sqrt{\tilde\mu^2+2\sigma^2\lambda}}+\frac{\sqrt{\tilde\mu^2+2\sigma^2\lambda}}{\sigma^2}\left(\ln\frac{x}{S_0}\right)\ln\left(\pm\frac{\sqrt{\tilde\mu^2+2\sigma^2\lambda}}{\sigma^2}\ln\frac{x}{S_0}\right)\right]
\quad(x\to S_0\pm).
\end{align*}
The second term vanishes in the $x\to S_0\pm$ limit.
Therefore, $f_{\Gam}(x)$ is differentiable at $x=S_0$, whose derivative becomes
\[
f_{\Gam}'(S_0)=\frac{4\lambda^{3/2}\sigma}{\sqrt{2}\pi(\tilde\mu^2+2\sigma^2\lambda)}\left(\frac{\tilde\mu}{\sigma^2}-1\right).
\]

When $k<3/2$,
\begin{align*}
f_{\Gam}'(x)&\simeq\frac{2\lambda^k}{\sqrt{2\pi}\Gamma(k)\sigma(\tilde\mu^2+2\sigma^2\lambda)^{(2k-1)/4}}
\frac{1}{x}\left(\frac{x}{S_0}\right)^{\tilde\mu/\sigma^2-1}\\
&\qquad\left[\frac{1}{2}\left(\frac{\tilde\mu}{\sigma^2}-1\right)\Gamma\left(k-\frac{1}{2}\right)\left(\frac{\sqrt{\tilde\mu^2+2\sigma^2\lambda}}{2\sigma^2}\right)^{-k+1/2}\right.\\
&\qquad\quad\left.\mp\Gamma\left(\frac{3}{2}-k\right)\left(\frac{\sqrt{\tilde\mu^2+2\sigma^2\lambda}}{2\sigma^2}\right)^{k-1/2}\left(\pm\ln\frac{x}{S_0}\right)^{2k-2} \right]\quad (x\to S_0\pm).
\end{align*}
The behavior of $f_{\Gam}'$ changes depending on the sign of $2k-2$.
For $k>1$ ($2k-2>0$), the second term vanishes as $x\to S_0\pm$, and thus, $f_{\Gam}(x)$ is differentiable at $x=S_0$.
For $k=1$, the $(\ln(x/S_0))^{2k-2}$ factor becomes unity, and
\[
f_{\Gam}'(S_0^{\pm})=\frac{\lambda}{S_0\sqrt{\tilde\mu^2+2\sigma^2\lambda}}\left(\frac{\tilde\mu}{\sigma^2}-1\mp\frac{\sqrt{\tilde\mu^2+2\sigma^2\lambda}}{\sigma^2}\right).
\]
The left and right derivatives at $x=S_0$ are both finite but different.
For $k<1$ ($2k-2<0$), the $(\pm\ln(x/S_0))^{2k-2}$ factor tends to infinity as $x\to S_0\pm$, and the derivative at $x=S_0$ diverges.
\end{proof}

\begin{remark}
For $1/2<k<1$, the value $f_{\Gam}(S_0)$ is finite but the derivative diverges; that is, $f_{\Gam}(x)$ has a cusp at $x=S_0$.
The case of $k=1$, corresponding to the double Pareto distribution, is placed at the boundary of undifferentiable $f_{\Gam}(x)$ at $x=S_0$.
Proposition~\ref{prop:gamma_diff} states that $f_{\Gam}(S_0)$ for $k>1$ is \textit{at least} of class $C^1$.
It is expected that the smoothness of $f_{\Gam}(x)$ increases for large $k$, but the boundary $k$ values between $C^1$ and $C^2$, $C^2$ and $C^3$, and so on have not yet been determined.
\end{remark}

\begin{figure}[t!]\centering
\begin{minipage}{60truemm}
\begin{flushleft}
\raisebox{36mm}{(a)}\hspace{-3mm}
\includegraphics[scale=0.8]{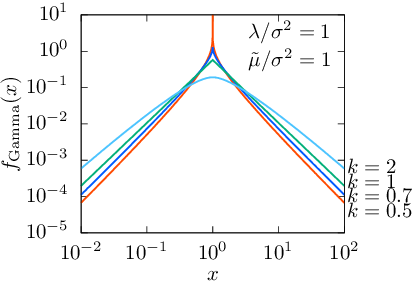}
\end{flushleft}
\end{minipage}
\begin{minipage}{65truemm}
\begin{flushleft}
\raisebox{36mm}{(b)}\hspace{-3mm}
\includegraphics[scale=0.8]{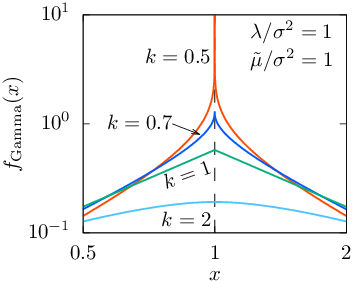}
\end{flushleft}
\end{minipage}
\vspace{0.5\baselineskip}
\begin{minipage}{60truemm}
\begin{flushleft}
\raisebox{36mm}{(c)}\hspace{-3mm}
\includegraphics[scale=0.8]{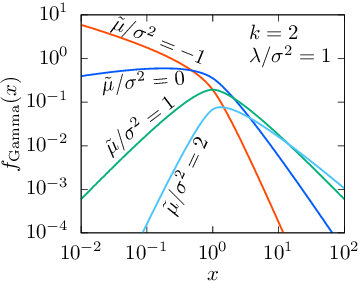}
\end{flushleft}
\end{minipage}
\begin{minipage}{65truemm}
\begin{flushleft}
\raisebox{36mm}{(d)}\hspace{-3mm}
\includegraphics[scale=0.8]{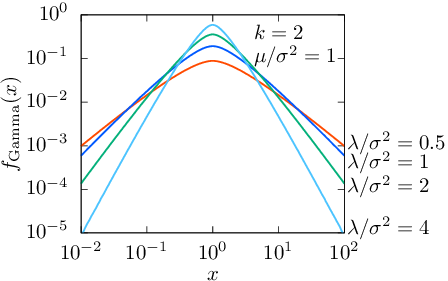}
\end{flushleft}
\end{minipage}
\caption{
Log-log graphs of $f_{\Gam}(x)$ for $S_0=1$ and different $k$, $\lambda/\sigma^2$, and $\tilde\mu/\sigma^2$.
(a) Dependence on $k$: $k=0.5$, $0.7$, $1$, and $2$ with $\lambda/\sigma^2=1$ and $\tilde\mu/\sigma^2=1$.
(b) Enlarged graph of (a) near $x=S_0=1$. The vertical dashed line indicates $x=S_0$.
(c) Dependence on $\tilde\mu/\sigma^2$: $\tilde\mu/\sigma^2=-1$, $0$, $1$, and $2$ with $k=2$ and $\lambda/\sigma^2=1$;
(d) Dependence on $\lambda/\sigma^2$: $\lambda/\sigma^2=0.5$, $1$, $2$, and $4$ with $k=2$ and $\tilde\mu/\sigma^2=1$.
}
\label{fig1}
\end{figure}

Figure~\ref{fig1} shows graphs of $f_{\Gam}(x)$ with $S_0=1$ on a log-log scale.
In Fig.~\ref{fig1}(a), the shape parameter $k$ of the gamma distribution is set to $k=0.5$, $0.7$, $1$, and $2$ with fixed $\lambda/\sigma^2=1$ and $\tilde\mu/\sigma^2=1$.
$f_{\Gam}(x)$ at large $x$ ($x\to\infty$) and small $x$ ($x\to0$) is large for large $k$.
Figure~\ref{fig1}(b) shows an enlarged graph of (a) near $x=S_0=1$.
The divergence at $S_0=1$ can be observed for $k=0.5$, as stated in Proposition~\ref{prop:gamma}(c).
Moreover, as stated in Proposition~\ref{prop:gamma_diff}, the graph for $k=0.7$ has a cusp, and the graph for $k=1$ is not differentiable at $x=S_0$.
The graph for $k=2$ appears smooth at $x=S_0$.
The graphs in Figs.~\ref{fig1}(a) and (b) are symmetric (on the logarithmic scale) across the axis $x=1(=S_0)$ because $f_{\Gam}(x)$ becomes a function of a single variable $|\ln(x/S_0)|$ when $\tilde\mu/\sigma^2=1$.
This symmetry breaks when $\tilde\mu/\sigma^2\ne1$, as shown in Fig.~\ref{fig1}(b) where $\tilde\mu/\sigma^2=-1$, $0$, $1$, and $2$ with fixed $k=2$ and $\lambda/\sigma^2=1$.
$f_{\Gam}(x)$ decreases rapidly as $x\to\infty$ for small $\tilde\mu/\sigma^2$, and conversely, it decreases rapidly as $x\to0+$ for large $\tilde\mu/\sigma^2$.
The graphs for $\tilde\mu/\sigma^2=0$, $1$, and $2$ appear to approach $0$ as $x\to0+$, whereas the graph for $\tilde\mu/\sigma^2=-1$ increases as $x\to0+$.
This observation is correct and is justified by Eq.~\eqref{eq:gamma_inequality}.
Lastly, the parameter$\lambda/\sigma^2$ changes the peak height at $x=S_0$; large $\lambda/\sigma^2$ yields a high peak and a rapid decrease, as shown in Fig.~\ref{fig1}(c).

\subsection{Observation time with inverse gamma distribution}
Another special family of the GIG distribution is the inverse gamma distribution.
By substituting $\nu=-k$, $\psi=0$, and $\chi=2\theta$ into Eq.~\eqref{eq:gamma_PDF}, we obtain
\begin{equation}
g_{\GIG}(t; -k,0,2\theta)=\frac{\theta^k}{\Gamma(k)}\left(\frac{1}{t}\right)^{k+1}\exp\left(-\frac{\theta}{t}\right).
\label{eq:igamma_PDF}
\end{equation}
This is the PDF of the inverse gamma distribution with the shape parameter $k$ and the scale parameter $\theta$, and we use $\InvGamma(k,\theta)$ for this distribution.
In deriving this PDF from Eq.~\eqref{eq:gamma_PDF}, the following integral is used:
\begin{equation}
\int_0^\infty u^{-k-1}\exp\left(-\frac{\theta}{2u}\right)du = \left(\frac{\theta}{2}\right)^{-k}\int_0^\infty v^{k-1}e^{-v}dv = \left(\frac{\theta}{2}\right)^{-k}\Gamma(k),
\label{eq:igamma_int}
\end{equation}
where we introduce $v=\theta/(2u)$.
The inverse gamma distribution is used in Bayesian modeling~\cite{Hoff} and survival distributions~\cite{Glen}.

When $T\sim\InvGamma(k,\theta)$, $1/T$ follows the gamma distribution $\Gam(k, 1/\theta)$.
The inverse gamma distribution includes the unshifted L\'evy distribution (L\'evy distribution with location parameter $0$) with $k=1/2$ and the inverse-chi-squared distribution with $\theta=1/2$.

\begin{corollary}
The PDF of $S_T$ with $T\sim\InvGamma(k,\theta)$ is
\begin{align}
f_{\InvGamma}(x)&=\frac{\sqrt{2}\theta^k}{\sqrt{\pi}\sigma\Gamma(k)S_0}\left(\frac{x}{S_0}\right)^{\tilde\mu/\sigma^2-1}\left(\frac{(\ln(x/S_0))^2}{\sigma^2}+2\theta\right)^{-(2k+1)/4}\left(\frac{\tilde\mu^2}{\sigma^2}\right)^{(2k+1)/4}\nonumber\\
&\quad K_{k+1/2}\left(\sqrt{\left(\frac{(\ln(x/S_0))^2}{\sigma^2}+2\theta\right)\frac{\tilde\mu^2}{\sigma^2}}\right)
\label{eq:igamma_f}
\end{align}
for $\tilde\mu\ne0$ and
\[
f_{\InvGamma}(x)=\frac{2^k\theta^k\Gamma(k+1/2)}{\sqrt{\pi}\sigma\Gamma(k)}\frac{1}{x}\left(\left(\frac{\ln(x/S_0)}{\sigma}\right)^2+2\theta\right)^{-(k+1/2)}
\]
for $\tilde\mu=0$.
\end{corollary}

\begin{proof}
In taking the $\psi\to0+$ limit of Theorem~\ref{thm2}, we use Proposition~\ref{prop:K}(a) and (c) in Appendix~\ref{apdx} to obtain
\[
\frac{\psi^{-k/2}}{K_{-k}(\sqrt{\psi\chi})}=\frac{\psi^{-k/2}}{K_{k}(\sqrt{\psi\chi})}\simeq\frac{\chi^{k/2}}{2^{k-1}\Gamma(k)},
\]
which is similar to Eq.~\eqref{eq:gamma_proof}.

The PDF $f_{\InvGamma}(x)$ for $\tilde\mu=0$ can be obtained by taking the $\tilde\mu\to0$ limit of Eq.~\eqref{eq:igamma_f} and applying Proposition~\ref{prop:K}(c), or it can be derived directly from Eq.~\eqref{eq:general_f} as
\begin{align*}
f_{\InvGamma}(x)&=\int_0^\infty f_{\LN}(x;\ln S_0, \sigma^2t)\frac{\theta^k}{\Gamma(k)}t^{-k-1}\exp\left(-\frac{\theta}{t}\right)dt\\
&=\frac{\theta^k}{\sqrt{2\pi}\sigma x\Gamma(k)}\int_0^\infty t^{-k-3/2}\exp\left[-\left(\frac{(\ln(x/S_0))^2}{2\sigma^2}+\theta\right)\frac{1}{t}\right]dt,
\end{align*}
and this integral is calculated as in Eq.~\eqref{eq:igamma_int}.
\end{proof}

As with $f_{\Gam}(x)$ (see Remark~\ref{remark:gamma}), the essential independent parameters of $f_{\InvGamma}(x)$ are $k$, $\tilde\mu/\sigma^2$, and $\theta\sigma^2$ (if $x/S_0$ is considered).
Unlike $f_{\Gam}(x)$, $f_{\InvGamma}(x)$ has no singularity at $x=S_0$ for any parameter values, and $f_{\InvGamma}(x)$ is always a smooth ($C^\infty$) function.

The moment-generating function of $T\sim\InvGamma(k,\theta)$ is
\[
M_T(s)=\frac{\theta^k}{\Gamma(k)}\int_0^\infty t^{-k-1}\exp\left(st-\frac{\theta}{t}\right)dt,
\]
which is defined for $s\le0$.
By calculating the integral, $M_T(0)=1$ and
\[
M_T(s)=\frac{2(-\theta s)^{k/2}}{\Gamma(k)} K_k(2\sqrt{-\theta s})
\]
for $s<0$~\cite{Paolella}.
Based on Theorem~\ref{thm1}, the $m$th moment of $S_T$ is finite when $m\tilde\mu+m^2\sigma^2/2\le0$.
Therefore, $\tilde\mu<0$ is a necessary condition for finite $E[S_T^m]$, and finite moment $E[S_T^m]$ is realized only at most $m\le 2|\tilde\mu|/\sigma^2$.
Compared with the gamma distributed $T\sim\Gam(k,\lambda)$ case, which can result in a finite moment for any order $m$, the condition for finite $E[S_T^m]$ under $T\sim\InvGamma(k,\theta)$ is severe.
This difference probably arises because the PDF of the inverse gamma distribution, Eq.~\eqref{eq:igamma_PDF}, has power-law-like decay for large $t$, which is significantly slower than the decay of the gamma distribution.
Another difference between the gamma and inverse gamma distributions is that $M_T(s)$ of $T\sim\InvGamma(k,\theta)$ involves the modified Bessel function, whereas that of $T\sim\Gam(k,\lambda)$ is expressed as an elementary function.

For example, the mean $E[S_T]$ becomes
\begin{equation}
E[S_T]=\frac{S_0 2^{1-k/2}}{\Gamma(k)}\theta^{k/2}(2|\tilde\mu|-\sigma^2)^{k/2} K_k\left(\sqrt{2(2|\tilde\mu|-\sigma^2)\theta}\right)
\label{eq:invgamma_mean}
\end{equation}
for $\sigma^2\le2|\tilde\mu|$.
This is a function of two parameters $k$ and $(2|\tilde\mu|-\sigma^2)\theta$, except for the proportionality constant $S_0$.
For $k=1/2$, where $T$ follows an unshifted L\'evy distribution, Proposition~\ref{prop:K}(d) in Appendix~\ref{apdx} simplifies $E[S_T]$ to
\[
E[S_T]=S_0\exp\left(-\sqrt{2(2|\tilde\mu|-\sigma^2)\theta}\right).
\]
Figure~\ref{fig2} shows a semi-log graph of $E[S_T]$ as a function of $(2|\tilde\mu|-\sigma^2)\theta$ with $k=0.5$, $1$, and $2$.
$E[S_T]$ decreases as $(2|\tilde\mu|-\sigma^2)\theta$ increases, and it becomes large for large $k$.
Using Proposition~\ref{prop:K}(c), $E[S_T]$ approaches $S_0$ as $(2|\tilde\mu|-\sigma^2)\theta\to0+$.

\begin{figure}\centering
\includegraphics[scale=0.8]{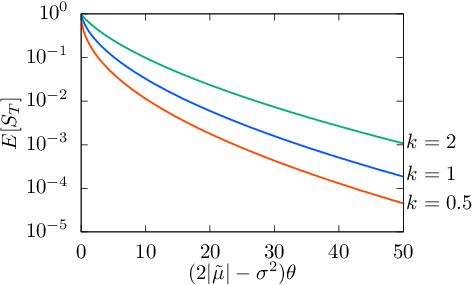}
\caption{
Semi-log graphs of the mean $E[S_T]$ given in Eq.~\eqref{eq:invgamma_mean} as a function of $(2|\tilde\mu|-\sigma^2)\theta$, with $k=0.5$, $1$, and $2$ and $S_0=1$.
}
\label{fig2}
\end{figure}

\begin{proposition}[Asymptotic form of $f_{\InvGamma}(x)$]
\begin{enumerate}
\renewcommand{\labelenumi}{$\mathrm{(\alph{enumi})}$}
\item For $\tilde\mu>0$,
\[
f_{\InvGamma}(x)\simeq\frac{(|\tilde\mu|\theta)^k}{\Gamma(k)}\frac{1}{x}\left|\ln\frac{x}{S_0}\right|^{-k-1}\quad(x\to\infty)
\]
and
\[
f_{\InvGamma}(x)\simeq\frac{(|\tilde\mu|\theta)^k}{\Gamma(k)S_0}\left(\frac{x}{S_0}\right)^{2\tilde\mu/\sigma^2-1}\left|\ln\frac{x}{S_0}\right|^{-k-1}\quad(x\to0+).
\]
\item For $\tilde\mu<0$,
\[
f_{\InvGamma}(x)\simeq\frac{(|\tilde\mu|\theta)^k}{\Gamma(k)S_0}\left(\frac{x}{S_0}\right)^{-2|\tilde\mu|/\sigma^2-1}\left|\ln\frac{x}{S_0}\right|^{-k-1}\quad(x\to\infty)
\]
and
\[
f_{\InvGamma}(x)\simeq\frac{(|\tilde\mu|\theta)^k}{\Gamma(k)}\frac{1}{x}\left|\ln\frac{x}{S_0}\right|^{-k-1}\quad(x\to0+).
\]
\item For $\tilde\mu=0$,
\[
f_{\InvGamma}(x)\simeq \frac{(2\theta\sigma^2)^k\Gamma(k+1/2)}{\sqrt{\pi}\Gamma(k)}\frac{1}{x}\left|\ln\frac{x}{S_0}\right|^{-2k-1}\quad(x\to\infty, x\to0+).
\]
$f_{\InvGamma}(x)$ for $\tilde\mu=0$ has the same asymptotic form in the $x\to\infty$ and $x\to0+$ limits.
\end{enumerate}
\end{proposition}

From this proposition, $f_{\InvGamma}(0+)$ is determined by the $\tilde\mu/\sigma^2$ value as follows:
\[
\lim_{x\to0+}f_{\InvGamma}(x)=
\begin{dcases}
0 & \dfrac{\tilde\mu}{\sigma^2}>\frac{1}{2},\\
\infty & \dfrac{\tilde\mu}{\sigma^2}\le\frac{1}{2}.
\end{dcases}
\]

\begin{figure}[t!]\centering
\raisebox{36mm}{(a)}\hspace{-3mm}
\includegraphics[scale=0.8]{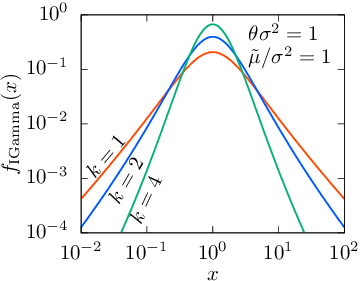}
\raisebox{36mm}{(b)}\hspace{-3mm}
\includegraphics[scale=0.8]{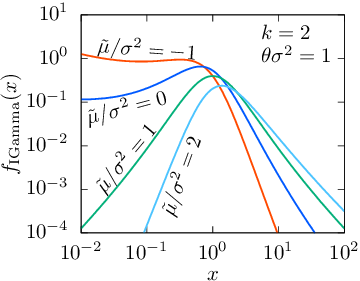}\\
\raisebox{36mm}{(c)}\hspace{-3mm}
\includegraphics[scale=0.8]{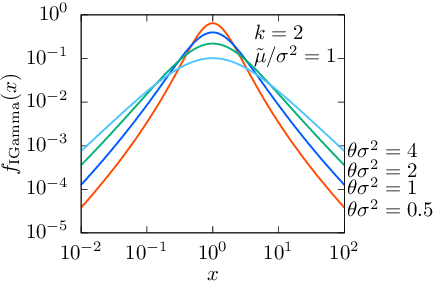}
\caption{
Log-log graphs of $f_{\InvGamma}(x)$ for $S_0=1$.
(a) Dependence on $k$: $k=1$, $2$, and $4$ with $\theta\sigma^2=1$ and $\tilde\mu/\sigma^2=1$;
(b) dependence on $\tilde\mu/\sigma^2$: $\tilde\mu/\sigma^2=-1$, $0$, $1$, and $2$ with $k=2$ and $\theta\sigma^2=1$;
(c) dependence on $\theta\sigma^2$: $\theta\sigma^2=0.5$, $1$, $2$, and $4$ with $k=2$ and $\tilde\mu/\sigma^2=1$.
}
\label{fig3}
\end{figure}

Log-log graphs of $f_{\InvGamma}(x)$ are shown in Fig.~\ref{fig3}.
As shown in Fig.~\ref{fig3}(a), $f_{\InvGamma}(x)$ for larger $k$ has a higher peak at $x=S_0=1$ and decreases rapidly as $x\to\infty$ and $x\to0+$.
This dependence is opposite to the tendency of the parameter $k$ in $f_{\Gam}(x)$, as shown in Fig.~\ref{fig1}(a).
The graphs in Fig.~\ref{fig3}(a) are symmetric across the $x=S_0$ axis because of $\tilde\mu/\sigma^2=1$ for which $f_{\InvGamma}(x)$ become a function of $(\ln(x/S_0))^2$, as with $f_{\Gam}(x)$.
As shown in Fig.~\ref{fig3}(c), increasing $\tilde\mu/\sigma^2$ makes the $x>S_0$ side of $f_{\InvGamma}(x)$ thicker and the $x<S_0$ side thinner, and vice versa.
Increasing $\theta\sigma^2$ results in a low peak at $x=S_0$ and a slow decrease, as shown in Fig.~\ref{fig3}(c).
This dependence is opposite to the tendency of the parameter $\lambda/\sigma^2$ in $f_{\Gam}(x)$ shown in Fig.~\ref{fig1}(d).

\subsection{Observation time with inverse Gaussian distribution}
By substituting $\nu=-1/2$, $\psi=\omega/\tau^2$, and $\chi=\omega$ ($\tau, \omega>0$) into the GIG distribution in Eq.~\eqref{eq:GIG_PDF}, we obtain the PDF of the inverse Gaussian distribution:
\[
g_{\GIG}\left(t; -\frac{1}{2}, \frac{\omega}{\tau^2}, \omega\right)=\sqrt{\frac{\omega}{2\pi t^3}}\exp\left(-\frac{\omega(t-\tau)^2}{2\tau^2 t}\right),
\]
where
\[
K_{-1/2}\left(\frac{\omega}{\tau}\right)=\sqrt{\frac{\pi\tau}{2\omega}}e^{-\omega/\tau}
\]
from Proposition~\ref{prop:K}(a) and (d) in Appendix~\ref{apdx}.
In this study, we use $\IG(\tau,\omega)$ for the inverse Gaussian distribution.
The parameter $\tau$ is the mean of this distribution, and $\omega$ is called the shape parameter.
The standard notation uses $\mu$ for the mean and $\lambda$ for the shape parameter~\cite{Chhikara}; however, we use $\tau$ and $\omega$ here because $\mu$ and $\lambda$ already denote other parameters.
The inverse Gaussian distribution appears in the hitting-time distribution of Brownian motion~\cite{Redner}.
The distribution $\IG(\tau, \omega)$ converges to an unshifted L\'evy distribution in the $\tau\to\infty$ limit.

As the inverse Gaussian distribution is a subclass of the GIG distribution, the PDF of $S_T$ under $T\sim\IG(\tau, \omega)$ can be derived as a special case of Theorem~\ref{thm2}.

\begin{corollary}
The PDF of $S_T$ for $T\sim\IG(\tau, \omega)$ becomes
\begin{align*}
f_{\IG}(x)&=\frac{\sqrt{\omega}e^{\omega/\tau}}{\pi\sigma S_0}\left(\frac{x}{S_0}\right)^{\tilde\mu/\sigma^2-1}\left(\frac{(\ln(x/S_0))^2}{\sigma^2}+\omega\right)^{-1/2}\left(\frac{\tilde\mu^2}{\sigma^2}+\frac{\omega}{\tau^2}\right)^{1/2}\\
&\quad K_{1}\left(\sqrt{\left(\frac{(\ln(x/S_0))^2}{\sigma^2}+\omega\right)\left(\frac{\tilde\mu^2}{\sigma^2}+\frac{\omega}{\tau^2}\right)}\right).
\end{align*}
\end{corollary}

Similar to $f_{\Gam}(x)$ and $f_{\InvGamma}(x)$, $f_{\IG}(x)$ has three reduced parameters, namely, $\tau\sigma^2$, $\omega\sigma^2$, and $\tilde\mu/\sigma^2$.
The function $f_{\IG}(x)$ is continuous and differentiable ($C^\infty$) over entire $x>0$ including $x=S_0$.

The moment-generating function of $T\sim\IG(\tau, \omega)$ becomes
\[
M_T(s)=\exp\left[\frac{\omega}{\tau}\left(1-\sqrt{1-\frac{2\tau^2 s}{\omega}}\right)\right],
\]
which is defined for $s<\omega/(2\tau^2)$.
Owing to Theorem~\ref{thm1}, the mean $E[S_T]$ becomes finite if and only if
\[
\tilde\mu+\frac{\sigma^2}{2}=\mu<\frac{\omega}{2\tau^2},
\]
and the mean becomes
\[
E[S_T]=\exp\left[\frac{\omega}{\tau}\left(1-\sqrt{1-\frac{2\tau^2\mu}{\omega}}\right)\right].
\]

\begin{proposition}[Asymptotic form of $f_{\IG}(x)$]
\[
f_{\IG}(x)\simeq\frac{\sqrt{\omega}e^{\omega/\tau}}{\sqrt{2\pi} S_0}\left(\tilde\mu^2+\frac{\omega\sigma^2}{\tau^2}\right)^{1/4}\left(\frac{x}{S_0}\right)^{(\tilde\mu-\sqrt{\tilde\mu^2+\omega\sigma^2/\tau^2})/\sigma^2-1}\left(\ln\frac{x}{S_0}\right)^{-3/2}\quad (x\to\infty)
\]
and
\[
f_{\IG}(x)\simeq\frac{\sqrt{\omega}e^{\omega/\tau}}{\sqrt{2\pi} S_0}\left(\tilde\mu^2+\frac{\omega\sigma^2}{\tau^2}\right)^{1/4}\left(\frac{x}{S_0}\right)^{(\tilde\mu+\sqrt{\tilde\mu^2+\omega\sigma^2/\tau^2})/\sigma^2-1}\left|\ln\frac{x}{S_0}\right|^{-3/2}\quad (x\to0+).
\]
\end{proposition}

From this proposition, $\lim_{x\to0+}f_{\IG}(x)=0$ holds if and only if
\[
\frac{\tilde\mu+\sqrt{\tilde\mu^2+\omega\sigma^2/\tau^2}}{\sigma^2}>1.
\]
This inequality is reduced to
\[
2\frac{\tilde\mu}{\sigma^2}+\frac{\omega}{\tau^2\sigma^2}>1.
\]

\begin{figure}[t!]\centering
\raisebox{36mm}{(a)}\hspace{-3mm}
\includegraphics[scale=0.8]{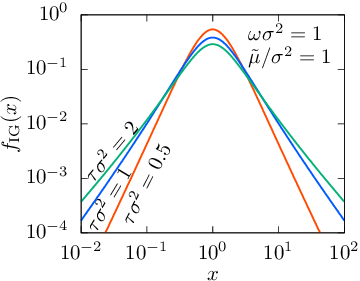}
\raisebox{36mm}{(b)}\hspace{-3mm}
\includegraphics[scale=0.8]{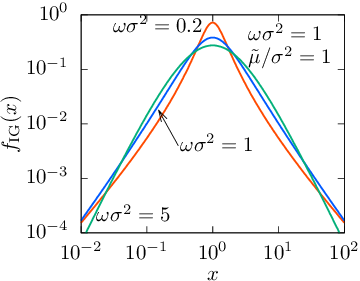}
\raisebox{36mm}{(c)}\hspace{-3mm}
\includegraphics[scale=0.8]{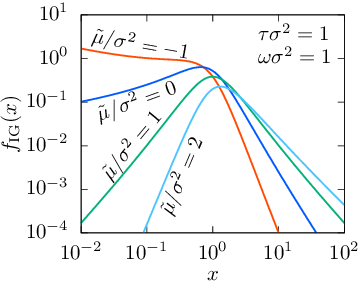}
\caption{
Log-log graphs of $f_{\IG}(x)$ for $S_0=1$.
(a) Dependence on $\tau\sigma^2$: $\tau\sigma^2=0.5$, $1$, and $2$ with $\omega\sigma^2=1$ and $\tilde\mu/\sigma^2=1$;
(b) dependence on $\omega\sigma^2$: $\omega\sigma^2=0.2$, $1$, and $5$ with $\tau\sigma^2=1$ and $\tilde\mu/\sigma^2=1$.
(c) dependence on $\tilde\mu/\sigma^2$: $\tilde\mu/\sigma^2=-1$, $0$, $1$, and $2$ with $\tau\sigma^2=1$ and $\omega\sigma^2=1$.
}
\label{fig4}
\end{figure}

Figure~\ref{fig4} shows the parameter dependence of $f_{\IG}(x)$ for $S_0=1$ on a log-log scale. 
The peak height at $x=S_0$ decreases for large $\tau\sigma^2$ and $\omega\sigma^2$, as shown in Figs.~\ref{fig4}(a) and (b), respectively.
In Fig.~\ref{fig4}(c), the parameter $\tilde\mu/\sigma^2$ changes the left-right symmetry across the axis $x=S_0$, as with $f_{\Gam}(x)$ in Fig.~\ref{fig1}(b) and $f_{\InvGamma}(x)$ in Fig.~\ref{fig3}(b).

A nontrivial relation between $f_{\InvGamma}(x)$ and $f_{\IG}(x)$ exists.
Consider $f_{\InvGamma}(x)$ having the parameters $\tilde\mu\ne0$, $\sigma^2$, $k$, and $\theta$.
We substitute $k=1/2$ and $\theta=\omega/2$ and introduce $\tau$ such that $\tau>\sigma\sqrt{2\theta}/|\tilde\mu|$.
Moreover, we introduce
\[
\tilde\mu'=\left(1-\frac{2\sigma^2\theta}{\tilde\mu^2\tau^2}\right)\tilde\mu,\quad
\sigma'^2=\left(1-\frac{2\sigma^2\theta}{\tilde\mu^2\tau^2}\right)\sigma^2,
\]
where $\sigma'^2>0$ is guaranteed by the choice of $\tau$.
Owing to
\[
\frac{\tilde\mu}{\sigma^2}=\frac{\tilde\mu'}{\sigma'^2},\quad
\frac{\tilde\mu^2}{\sigma^2}=\frac{\tilde\mu'^2}{\sigma'^2}+\frac{\omega}{\tau^2},
\]
$f_{\InvGamma}(x)$ becomes $f_{\IG}(x)$ with $T\sim\IG(\tau,\omega)$ and $(\tilde\mu, \sigma^2)$ replaced with $(\tilde\mu', \sigma'^2)$.
Although there exists no direct transformation between the inverse Gamma distribution and the inverse Gaussian distribution, we can construct the relation between PDFs $f_{\InvGamma}(x)$ and $f_{\IG}(x)$ of $S_T$ involving the change of $\tilde\mu$ and $\sigma^2$.

In conclusion, we have generalized the framework of geometric Brownian motion observed at an exponentially distributed time, which yields the double Pareto distribution having a power-law tail, by replacing the exponential distribution with more general ones.
Theorem~\ref{thm1} provides a formula for moments of the observed geometric Brownian motion under an arbitrary observation time distribution.
Furthermore, we have shown that the generalized inverse Gaussian (GIG) distribution forms a class of observation time distributions for which the PDF of the observed geometric Brownian motion can be calculated in a closed form (Theorem~\ref{thm2}).
The GIG distribution includes the exponential, gamma, inverse gamma, and inverse Gaussian distributions as special cases, and Theorem~\ref{thm2} can be regarded as a generalization of the double Pareto distribution.

\appendix
\section{Properties and formulas for the modified Bessel function of the second kind}\label{apdx}
The useful properties of the modified Bessel function of the second kind, $K_\nu(z)$, are presented.
Here, we define the relation
\[
f(z)\simeq g(z)\quad(z\to a)
\]
as
\[
\lim_{z\to a}\frac{f(z)}{g(z)}=1.
\]

\begin{proposition}[Properties of $K_\nu(z)$~\cite{Olver}]
\begin{enumerate}
\renewcommand{\labelenumi}{$\mathrm{(\alph{enumi})}$}
\item $K_\nu(z)=K_{-\nu}(z)$.
\item Asymptotic form in $z\to\infty$:
\[
K_\nu(z)\simeq \sqrt{\frac{\pi}{2z}}e^{-z}\quad (z\to\infty).
\]
\item Asymptotic form in $z\to0+$:
\[
K_\nu(z)\simeq
\begin{dcases}
\frac{1}{2}\Gamma(\nu)\left(\frac{z}{2}\right)^{-\nu} & \nu>0,\\
-\ln z & \nu=0
\end{dcases}
\quad (z\to0+).
\]
\item For $\nu=1/2$,
\[
K_{1/2}(z)=\sqrt{\frac{\pi}{2z}}e^{-z}.
\]
\item Derivative:
\[
\frac{dK_\nu(z)}{dz}=-\frac{\nu}{z}K_\nu(z)-K_{\nu-1}(z).
\]
\end{enumerate}
\label{prop:K}
\end{proposition}

\begin{remark}\label{remark:K}
According to Olver et al.~\cite{Olver}, the modified Bessel function of the second kind $K_\nu(z)$ is characterized as satisfying the modified Bessel differential equation
\begin{equation}
z^2\frac{d^2 K_\nu(z)}{dz^2}+z\frac{d K_\nu(z)}{dz}-(z^2+\nu^2)K_\nu(z)=0
\label{eq:Bessel1}
\end{equation}
and having the asymptotic form
\begin{equation}
K_\nu(z)\simeq \sqrt{\frac{\pi}{2z}}e^{-z}\quad (z\to\infty).
\label{eq:Bessel2}
\end{equation}
By adopting this characterization, we can provide a brief proof of Lemma~\ref{lemma1}.

\begin{proof}[Proof of Lemma~\ref{lemma1} assuming Eqs.~\eqref{eq:Bessel1} and \eqref{eq:Bessel2} as the characterization of $K_\nu(z)$]
We set
\begin{equation}
\mathcal{K}_\nu(z)= 2^{\nu-1}\int_0^\infty t^{\nu-1}\exp\left(-zt-\frac{z}{4t}\right)dt.
\label{eq:intBessel}
\end{equation}
Notably, Lemma~\ref{lemma1} is equivalent to $\mathcal{K}_\nu(z)=K_\nu(z)$.
In fact, $\mathcal{K}_\nu(z)=K_\nu(z)$ is obtained by substituting $a=p=z$ in Lemma~\ref{lemma1}.
Conversely, assuming $\mathcal{K}_\nu(z)=K_\nu(z)$, Lemma~\ref{lemma1} is obtained by setting $z=\sqrt{ap}$ and introducing the substitution $t=\sqrt{a/p}u$ in Eq.~\eqref{eq:intBessel}.
Thus, to prove Lemma~\ref{lemma1}, it suffices to demonstrate that the function $\mathcal{K}_\nu(z)$ satisfies the modified Bessel equation~\eqref{eq:Bessel1} and the asymptotic form~\eqref{eq:Bessel2}.

First, we show that $\mathcal{K}_\nu(z)$ satisfies the modified Bessel equation~\eqref{eq:Bessel1}.
The derivative of $\mathcal{K}_\nu(z)$ becomes
\begin{align*}
\frac{d \mathcal{K}_\nu(z)}{dz}&=2^{\nu-1}\int_0^\infty t^{\nu-1}\left(-t-\frac{1}{4t}\right)\exp\left(-zt-\frac{z}{4t}\right)dt\\
&=\frac{2^{\nu-1}}{z}\int_0^\infty t^{\nu}\left(-z+\frac{z}{4t^2}\right)\exp\left(-zt-\frac{z}{4t}\right)dt-2^{\nu-2}\int_0^\infty t^{\nu-2}\exp\left(-zt-\frac{z}{4t}\right)dt\\
&=\frac{2^{\nu-1}}{z}\int_0^\infty t^\nu \frac{d}{dt}\exp\left(-zt-\frac{z}{4t}\right) dt-2^{\nu-2}\int_0^\infty t^{\nu-2}\exp\left(-zt-\frac{z}{4t}\right)dt\\
&=-\frac{\nu}{z}\mathcal{K}_\nu(z)-\mathcal{K}_{\nu-1}(z),
\end{align*}
where the first integral is calculated using integration by parts in the last equality.
Then, we obtain
\[
z^2\frac{d^2\mathcal{K}_\nu(z)}{dz^2}+z\frac{d\mathcal{K}_\nu(z)}{dz}
=\nu^2\mathcal{K}_\nu(z)+z(\nu-1)\mathcal{K}_{\nu-1}(z)-z^2\frac{d\mathcal{K}_{\nu-1}(z)}{dz}.
\]
The relation $\mathcal{K}_{\nu-1}(z)=\mathcal{K}_{1-\nu}(z)$ (corresponding to Proposition~\ref{prop:K}(a)) holds because
\[
\mathcal{K}_{\nu-1}(z)=2^{(1-\nu)-1}\int_0^\infty u^{(1-\nu)-1}\exp\left(-zu-\frac{z}{4u}\right)du=\mathcal{K}_{1-\nu}(z),
\]
where $u=(4t)^{-1}$ is introduced.
By using this property, the derivative of $\mathcal{K}_{\nu-1}(z)$ can be expressed as
\[
\frac{d\mathcal{K}_{\nu-1}(z)}{dz}=\frac{d\mathcal{K}_{1-\nu}(z)}{dz}=-\frac{1-\nu}{z}\mathcal{K}_{1-\nu}(z)-\mathcal{K}_{-\nu}(z)
=\frac{\nu-1}{z}\mathcal{K}_{\nu-1}(z)-\mathcal{K}_\nu(z).
\]
Therefore, $\mathcal{K}_\nu(z)$ satisfies the modified Bessel differential equation:
\[
z^2\frac{d^2\mathcal{K}_\nu(z)}{dz^2}+z\frac{d\mathcal{K}_\nu(z)}{dz}=(z^2+\nu^2)\mathcal{K}_\nu(z).
\]

Next, the asymptotic form as $z\to\infty$ is obtained using Laplace's method~\cite{Bender}.
The function $h(t)=-t-(4t)^{-1}$ takes the maximum value $h(1/2)=-1$ at $t=1/2$ and $h''(1/2)=-4$.
Therefore,
\[
\mathcal{K}_\nu(z)=2^{\nu-1}\int_0^\infty t^{\nu-1}e^{zh(t)}dt
\simeq 2^{\nu-1}\sqrt{\frac{2\pi}{z|h''(1/2)|}}\left(\frac{1}{2}\right)^{\nu-1}e^{zh(1/2)}
=\sqrt{\frac{\pi}{2z}}e^{-z}.
\]

Thus, $\mathcal{K}_\nu(z)$ defined in Eq.~\eqref{eq:intBessel} is proved to be identical to the modified Bessel function of the second kind $K_\nu(z)$.
\end{proof}
\end{remark}

\section*{Acknowledgments}
The authors are grateful to an anonymous reviewer for suggesting the relation to subordinated stochastic processes.
This study was supported by a Grant-in-Aid for Scientific Research (C) 19K03656 from Japan Society for the Promotion of Science.

\section*{Data availability}
Data sharing is not applicable to this article as no data sets were created or analyzed in this study.

\bibliography{yamamoto}

\end{document}